\documentclass[10pt]{amsart}

\usepackage[leqno]{amsmath}
\usepackage{amsthm}
\usepackage{amsfonts}
\usepackage{amssymb}
\usepackage{eucal}
\usepackage[all]{xy}

\DeclareFontFamily{OT1}{rsfs}{}
\DeclareFontShape{OT1}{rsfs}{n}{it}{<-> rsfs10}{}
\DeclareMathAlphabet{\mathscr}{OT1}{rsfs}{n}{it}

\theoremstyle{plain}
 \newtheorem{prop}{Proposition}[section]
 \newtheorem{lem}[prop]{Lemma}
 
 \newtheorem{conj}[prop]{Conjecture}
 \newtheorem{thm}[prop]{Theorem}

\theoremstyle{definition}

\theoremstyle{remark}
  
  \newtheorem{rem}[subsection]{Remark}

\numberwithin{equation}{section}

   \DeclareMathOperator{\ord}{ord}
   
\title{Upper bounds on the solutions to $n = p+m^2$}
\author{Aran Nayebi}
\email{aran.nayebi@gmail.com}

\subjclass[2000]{Primary 11P32; Secondary 11P55}
\keywords{Additive, Conjecture H, circle method}
\begin{document}
\begin{abstract}
Hardy and Littlewood conjectured that every large integer $n$ that is not a square is the sum of a prime and a square. They believed that the number $\mathcal{R}(n)$ of such representations for $n = p+m^2$ is asymptotically given by
\begin{equation*}
\mathcal{R}(n) \sim \frac{\sqrt{n}}{\log n}\prod_{p=3}^{\infty}\left(1-\frac{1}{p-1}\left(\frac{n}{p}\right)\right),
\end{equation*}
where $p$ is a prime, $m$ is an integer, and $\left(\frac{n}{p}\right)$ denotes the Legendre symbol. Unfortunately, as we will later point out, this conjecture is difficult to prove and not \emph{all} integers that are nonsquares can be represented as the sum of a prime and a square. Instead in this paper we prove two upper bounds for $\mathcal{R}(n)$ for $n \le N$. The first upper bound applies to \emph{all} $n \le N$. The second upper bound depends on the possible existence of the Siegel zero, and assumes its existence, and applies to all $N/2 < n \le N$ but at most $\ll N^{1-\delta_1}$ of these integers, where $N$ is a sufficiently large positive integer and $0<\delta_1 \le 0.000025$.
\end{abstract}

\maketitle

\def\ord{{\mathrm{ord}}}
\def\scr{\scriptstyle}
\def\\{\cr}
\def\[{\left[}
\def\]{\right]}
\def\<{\langle}
\def\>{\rangle}
\def\fl#1{\left\lfloor#1\right\rfloor}
\def\rf#1{\left\lceil#1\right\rceil}
\def\lcm{{\rm lcm\/}}

\def\C{\mathbb C}
\def\R{\mathbb R}
\def\Q{{\mathbb Q}}
\def\F{{\mathbb F}}
\def\Z{{\mathbb Z}}
\def\cO{{\mathcal O}}

\def\ord{{\mathrm{ord}}}
\def\Nm{{\mathrm{Nm}}}
\def\L{{\mathbb L}}

\def\xxx{\vskip5pt\hrule\vskip5pt}
\def\yyy{\vskip5pt\hrule\vskip2pt\hrule\vskip5pt}

\section{Introduction}
\subsection{Background Information}
In this paper, we consider the following conjecture of Hardy and Littlewood \cite{1}:
\begin{conj}[Hardy-Littlewood; Conjecture H]\label{HL}
Every large integer $n$ that is not a square is the sum of a prime and a square. The number $\mathcal{R}(n)$ of representations for $n = p + m^2$ is given asymptotically by
\begin{equation} \label{1}
\mathcal{R}(n) \sim \mathscr{P}(n)\frac{\sqrt{n}}{\log n}
\end{equation}
\begin{equation}
\mathscr{P}(n) = \prod_{p=3}^{\infty}\left(1-\frac{1}{p-1}\left(\frac{n}{p}\right)\right),
\end{equation}
where $p$ is a prime, $m$ is an integer, and $\left(\frac{n}{p}\right)$ denotes the Legendre symbol.
\end{conj}
In 1937, Davenport and Heilbronn \cite{3} proved that Conjecture~\ref{HL} holds for \emph{almost} all natural numbers. In fact, they showed that if we define the exceptional set as
\begin{eqnarray} \label{3}
E(N)&:=& \{n \le N: \text{$n$ is neither a square}\nonumber\\
&&\,\, \text{ nor the sum of a prime and a square}\},
\end{eqnarray}
then
\begin{equation} \label{4a}
|E(N)| \ll N \log^{c} N
\end{equation}
for some $c < 0$. In 1968, Miech \cite{6} proved that \eqref{4a} holds for arbitrary $c < 0$. Using the approach of Hardy and Littlewood \cite{1} via theta-functions, Vinogradov \cite{14} proved that there exist \emph{effectively computable} constants $c < 1$ and $\gamma > 0$ such that $|E(N)| \le \gamma N^{c}$. Br\"{u}nner, Perelli, and Pintz \cite{2} used the methods of Montgomery and Vaughan \cite{8} to prove the same result. Polyakov \cite{5} independently demonstrated that the results of Vinogradov \cite{14} hold without using Siegel's theorem, thereby enabling him to attain effective results. Basing his methods on those of \cite{2}, Wang \cite{9} made the computation of $c$ more rigorous and proved,
\begin{equation}
|E(N)| \ll N^{0.99}.
\end{equation}
The exponent was subsequently improved by Li \cite{10} to 0.9819. \newline
\indent Some work has been done in an attempt to the verify the asymptotic formula for $\mathcal{R}(n)$ in Conjecture~\ref{HL}. Miech \cite{6} proved that
\begin{equation} \label{Miech}
\mathcal{R}(n) = \mathscr{P}(n)\frac{\sqrt{n}}{\log n}\left(1+O\left(\frac{\log \log n}{\log n}\right)\right)
\end{equation}
holds for all but $O(N(\log N)^{A})$ positive integers $n \le N$ with any fixed $A < 0$. Polyakov also attempted to make progress on Conjecture~\ref{HL} \cite{4} \cite{5}. For all but $\ll N\cdot\exp\{-c\sqrt{\log N}\}$ integers $n \le N$, he obtained the following
\begin{equation} \label{5}
\mathcal{R}(n) = \mathscr{P}(n)\frac{\sqrt{n}}{\log n}\left(1+O\left(\exp\Big\{-\frac{\sqrt{\log n}}{\log^3\log n}\Bigr\}\right)\right).
\end{equation}
Unfortunately, a mistake occurs in one of Polyakov's estimates \cite{4}, and ``due to the possible existence of the Siegel zero, such a result is unlikely to be provable in the present state of knowledge'' \cite[pp. 347-8]{2}. \newline
\indent In this paper, the first upper bound we prove for $\mathcal{R}(n)$ holds for \emph{all} $n \le N$ ($N$ sufficiently large) with no exceptions. The second upper bound is achieved by way of Polyakov's \cite{5} methods, although he uses his methods for the entirely different purpose of determining the cardinality of the exceptional set $E(N)$. This second upper bound assumes the possible existence of the Siegel zero, and consequently has an exceptional set. However, before we present the main results of this paper, we first define some nomenclature.
\subsection{Notation}
We will use some of the same notation used by Polyakov \cite{5} for simplicity: \newline
\indent Suppose $m$ and $u$ are natural numbers; $p$ is a prime; $N$ is a sufficiently large positive integer; $\mu$ is the M\"{o}bius function; $\varphi$ is Euler's totient function; $\mathscr{P}(n):= \prod_{p=3}^{\infty}\left(1-\frac{1}{p-1}\left(\frac{n}{p}\right)\right)$ where $(n/p)$ is the Legendre symbol; $0 < \delta \le 0.0025$; $0 < \delta_{1} \le 0.000025$; $\epsilon > 0$; $Q=N^{\epsilon\delta}$; $\tau = N^{1-46\delta}$; $s = \sigma + it$ is a complex variable; $c,c_1,c_2, \ldots$ are absolute positive constants; $\chi$ is a Dirichlet character $\mod q$; $\chi_{0}$ is the principal character $\mod q$; $\chi^{*}$ is the primitive character corresponding to $\chi$; $\sum_{\chi}$ is the summation over all characters $\mod q$; $\sum_{a\le q}{^{*}}$ is the summation over a reduced system of residues $\mod q$; $L(s,\chi) = \sum_{n=1}^{\infty}\frac{\chi(n)}{n^s}$ is the $L$-function defined for $\sigma>1$; $\beta$, also known as a Siegel zero, is an exceptional real zero (if it exists) in the region $\Re(s) \ge c \log^{-1} Q$ for the $L$-function $L(s,\tilde{\chi})$ with the real primitive character $\tilde{\chi} \mod \tilde{r}$ where $\tilde{r} \le Q = \exp\{\log^{1/2} N\}$; $\epsilon(\beta) = \epsilon(\beta,Q)$ is a function equal to 1 if $\beta$ exists and is equal to 0 if otherwise; $\alpha$ and $x$ are real variables; $e(x) = \exp\{2\pi i x\}$; $B$ is a bounded quantity whose absolute value is bounded above by some constant that is independent of $n$ and $N$.
\subsection{Main Theorem}
Now we are ready to state the central results of this paper.
\begin{thm} \label{result1}
For all $n \le N$,
\begin{equation} \label{A}
\mathcal{R}(n) \le 2\cdot\mathscr{P}(n)\frac{N}{\log N}\left(1+O\left(\frac{\log\log 3N}{\log N}\right)\right).
\end{equation}
For all $N/2 < n \le N$ except for at most $\ll N^{1-\delta_1}$ of these integers, if the Siegel zero $\beta$ does indeed exist,
\begin{equation} \label{6}
\mathcal{R}(n) \ll_{N} \mathscr{P}(n)N^{1/2}\exp\Big\{\frac{-c}{\delta}\Bigr\}(1-\beta)\log N.
\end{equation}
If the Siegel zero does not exist, then the $(1-\beta)\log N$ term in \eqref{6} is removed and the resultant bound holds for all $N/2 < n \le N$ except for at most $\ll N^{1-\delta}$ of these integers. Note that $\delta$ and $\delta_1$ are fixed, where $0 < \delta \le 0.0025$ and $0 < \delta_{1} \le 0.000025$, as defined in \emph{\S 1.2}.
\end{thm}
\indent Although we have not proven Conjecture~\ref{HL} (which is unlikely to be proven in the current state of knowledge anyway), our results are of interest.
The upper bounds in Theorem~\ref{result1} involve $\mathscr{P}(n)$ (just like Conjecture~\ref{HL} does), the treatment of the conjectural Siegel zero is explicit, the cardinality of the exceptional set in \eqref{6} is rather small since it is contained within the cardinality of the exceptional set given by Br\"{u}nner, Perelli, and Pintz \cite{2} (see our \S 1.1). Moreover, under the likely assumption that the Siegel zero $\beta$ does not exist, the upper bound in \eqref{A} has no exceptions.

\section{Preliminaries}
In order to prove Theorem~\ref{result1}, we introduce some auxiliary functions and lemmas about these functions. We should inform the reader that any lemma that is presented without proof in this paper means that it has already been stated and proven by Polyakov \cite{5}.
Put
\begin{equation} \label{8}
\begin{split}
P(\alpha)&:=\sum_{Q<p\le N}\log pe(p\alpha) \\
F(\alpha)&:=\sum_{{\sqrt{N}/2}<n\le \sqrt{N}}e(m^2\alpha) \\
R(n)&:={\sum_{\substack{Q < p \le N\\{n=p+m^2}}}\sum_{\sqrt{N}/2 < m \le \sqrt{N}}}\log p.
\end{split}
\end{equation}
Note that $R(n)$ and $\mathcal{R}(n)$ are easily related by partial summation. An upper bound of $\ll$-type on $R(n)$ is applicable to $\mathcal{R}(n)$.
Thus, for $N/2 < n \le N$,
\begin{equation} \label{9}
R(n) = \int_{0}^{1}P(\alpha)F(\alpha)e(-n\alpha)d\alpha = \int_{-1/\tau}^{1-1/\tau} P(\alpha)F(\alpha)e(-n\alpha)d\alpha.
\end{equation}
\indent Dirichlet's approximation theorem leads us to the notion that each
$\alpha \in [-1/\tau, 1-1/\tau]$ can be represented in the form $\alpha=\frac{a}{q}+z$ for $1 \le q \le z$, $\gcd(a,q)=1$, and $|z| \le \frac{1}{q\tau}$.
Let
\begin{equation} \label{10}
\begin{split}
M_1&:=\{\alpha\in[-1/\tau, 1-1/\tau]: \text{for which~} q \le Q\\
&\qquad \text{ is in the indicated representation}\} \\
M_2&:=\{\alpha\in[-1/\tau, 1-1/\tau]: \text{for which~} q \le Q\\
&\qquad \text{ is not in the indicated representation}\}.
\end{split}
\end{equation}
\indent Now, put $R(n) = R_1(n) + R_2(n)$, where
\begin{equation} \label{11}
\begin{split}
R_1(n) = \int_{M_1}P(\alpha)F(\alpha)e(-n\alpha)d\alpha \\
R_2(n) = \int_{M_2}P(\alpha)F(\alpha)e(-n\alpha)d\alpha.
\end{split}
\end{equation}
\begin{lem} \label{lemma1}
For all $N/2 < n \le N$ except for $\ll N^{1-\delta}$ integers $n$,
\begin{equation} \label{12}
R_2(n) \ll n^{1/2-2\delta}.
\end{equation}
\end{lem}
\begin{proof}
By using Parseval's identity to show that
\begin{equation} \label{13}
\sum_{n \le N}R_2^2(n) \le \max_{M_1}|F(\alpha)|^2 \int_0^1 |P(\alpha)|^2 d\alpha
\end{equation}
where
\begin{equation} \label{14}
\int_0^1 |P(\alpha)|^2 d\alpha= \sum_{Q<p\le N}\log^2 N \ll N \log N,
\end{equation}
Polyakov \cite{5} proves that
\begin{equation} \label{polyakovparsv}
\sum_{n\le N}R_2^2(n) \le N^{2-5\delta},
\end{equation}
which implies our lemma.
\end{proof}
\indent As Polyakov \cite{5} mentions, the following lemma is due to Karatsuba \cite[Ch. IX, Sec. 2]{7}.
\begin{lem} \label{lemma2}
There exists a constant $c > 0$ such that $L(s,\chi^*) \ne 0$ for $\sigma \ge 1-c\log^{-1}Q$ and for all primitive characters $\chi^* \mod r$, where $r \le Q$ and $Q \ge 2$, with the possible exception of at most one primitive character $\tilde{\chi} \mod \tilde{r}$. If this character exists, then it is a quadratic character and the unique Siegel zero $\beta$ for the $L$-function $L(s, \tilde{\chi})$ satisfies
\begin{equation} \label{15}
c_1{{\tilde{r}}^{-1/2}}\log^{-2} \tilde{r} \le 1-\beta \le c \log^{-1} Q.
\end{equation}
Also, if there are any $L(s,\chi)$, where $\chi$ is a real character $\mod q$, such that $L(\beta,\chi) = 0$ in \eqref{15}, then $q \equiv 0\pmod{\tilde{r}}$.
\end{lem}
Next, put
\begin{equation} \label{16}
\begin{split}
V(a,q) &:= \sum_{1 \le m \le q} e\left(m^2 \frac{a}{q}\right) \\
K(z) &:= \sum_{N/4 < m \le N} \frac{e(mz)}{2\sqrt{m}}.
\end{split}
\end{equation}
Thus, $R_1(n) = R_1^{(1)}(n) + R_1^{(2)}(n) + R_1^{(3)}(n)$ if
\begin{equation*} \label{17}
\begin{split}
R_1^{(1)}(n)& =	\Bigg( \sum_{q \le Q}\frac{\mu(q)}{q\varphi(q)} \sum_{a \le q}{^{*}}V(a,q)e\left(-n \frac{a}{q}\right) \Bigg)\\
			&	\times \Bigg( \int_{-1/(q\tau)}^{1/(q\tau)} T(z,1)K(z)e(-nz)dz \\
            & \quad   - \epsilon(\beta)
			 	\sum_{\substack{q \le Q \\ q \equiv 0\pmod{ \tilde{r}}}} \frac{\tau(\tilde{\chi}\chi_{0})}{q\varphi(q)}
				\sum_{a \le q}{^{*}}V(a,q)\tilde{\chi}(a) \Bigg)\\
			&	\times \Bigg( e\left(-n \frac{a}{q}\right)\int_{-1/(q\tau)}^{1/q(\tau)}T(z,\beta)K(z)e(-nz)dz \Bigg)
\end{split}
\end{equation*}
\begin{equation*}
\begin{split}
R_1^{(2)}(n)&= \sum_{q \le Q}\sum_{a \le q}{^{*}}\int_{-1/(q\tau)}^{1/(q\tau)}P\left(\frac{a}{q}+z\right)\left(F\left(\frac{a}{q}+z\right)-\frac{V(a,q)}{q}K(z)\right)\\
& \quad e\left(-n\left(\frac{a}{q}+z\right)\right)dz
\end{split}
\end{equation*}
\begin{equation*}
\begin{split}
R_1^{(3)}(n) &= \sum_{q \le Q}\frac{1}{q\varphi(q)}\sum_{a \le q}{^{*}}V(a,q)e\left(-n \frac{a}{q}\right)\sum_{\chi}\chi(a)\tau(\tilde{\chi})\\
& \quad \times \int_{-1/(q\tau)}^{1/(q\tau)}K(z)W(\chi,z)e(-nz)dz,
\end{split}
\end{equation*}
where, as defined by Montgomery and Vaughan \cite{8},
\begin{equation*}
T(z,\gamma) = \sum_{Q<u \le N}u^{\gamma-1}e(uz)
\end{equation*}
and
\begin{equation*}
W(\chi,z) = \sum_{Q<p \le N}\chi(p)\log p e(pz)
\end{equation*}
in which $\chi \ne \chi_0$ and $\chi \ne \tilde{\chi}\chi_0$, and $\tau(\chi)$ is the Gauss sum. \newline
\indent In order to develop an upper bound for $R_1(n)$ we need three more lemmas.
\begin{lem} \label{lemma3}
For $N/2 < n \le N$,
\begin{align} \label{18}
&R_1^{(1)}(n)\nonumber\\
& = \Sigma(n,Q)\sum_{n=u+m}\frac{1}{2\sqrt{m}}-\epsilon(\beta)\Sigma(n,Q,\beta)\sum_{n=u+m}\frac{u^{\beta-1}}{2\sqrt{m}}+BN^{1/2-34\delta}\nonumber\\
\end{align}
where
\begin{equation} \label{19}
\Sigma(n,Q) = \sum_{q \le Q}\frac{\mu(q)}{q\varphi(q)}\sum_{a \le q}{^{*}}V(a,q)e\left(-n\frac{a}{q}\right)
\end{equation}
\begin{equation} \label{20}
\Sigma(n,Q,\beta)=\sum_{\substack{q \le Q \\ q \equiv 0\pmod{\tilde{r}}}}\frac{\tau(\tilde{\chi}\chi_0)}{q\varphi(q)}\sum_{a \le q}{^{*}}V(a,q)\tilde{\chi}(a)e\left(-n\frac{a}{q}\right).
\end{equation}
\end{lem}
\begin{lem} \label{lemma4}
It follows from the relation given in \eqref{14} that for all $N/2 < n \le N$,
\begin{equation} \label{21}
R_1^{(2)}(n) \ll N \sum_{q \le Q}\sum_{a \le q}{^{*}}\int_{-1/(q\tau)}^{1/(q\tau)}\left(q+\frac{N}{q\tau}\right)dz.
\end{equation}
\end{lem}
\begin{lem} \label{lemma5}
Put $q = kr$ where $\gcd(k,r)=1$ \emph{(}because otherwise $R_1^{(3)}(n)$ would be $0$\emph{)}, then for all $N/2 < n \le N$,
\begin{equation} \label{22}
\begin{split}
R_1^{(3)}(n)& = \Bigg( \sum_{r \le Q}\frac{1}{r\varphi(r)}\sum_{{\chi^*}\mod r}\tau(\overline{\chi^{*}})\sum_{a_{1}\le r}{^{*}}V(a_1,r) \Bigg) \\
			& \cdot \Bigg( \chi^{*}(a_{1})e\left(-n\frac{a_1}{r}\right)\int_{-1/r\tau}^{1/r\tau}K(z)W(\chi^{*},z)e(-nz)dz \Bigg) \\
			& \cdot \Bigg( \sum_{\substack{k \le Q/r \\ \gcd(k,r)=1}}\frac{\mu(k)}{k\varphi(k)}\sum_{a \le k}{^{*}}V(a,k)e\left(-n\frac{a}{k}\right)+BN^{1/2-2\delta} \Bigg).
\end{split}
\end{equation}
\end{lem}
The last set of two lemmas deals with auxiliary functions that will later be useful in developing an upper bound for $R(n)$.
\begin{lem} \label{lemma6}
Let $\mathscr{P}(n)$ be defined as in \emph{\S 1.2}, then for all $n \le N$, except for $\ll N^{0.7}$ of these integers,
\begin{equation} \label{23}
\Sigma(n,Q) = \mathscr{P}(n)+Bn^{-2\delta_1}.
\end{equation}
\end{lem}
\begin{lem} \label{lemma7}
If $r \le Q$, then for all $n \le N$, except for $\ll N^{1-\delta_1}$ of these integers, $\gcd(t,n) \le N^{1-\delta_1}$.
\end{lem}
Now we are ready to prove our main theorem.

\section{Proof of Theorem~\ref{result1}}
Since $R(n) = R_1(n)+R_2(n)$ and in Lemma~\ref{lemma1} we prove an upper bound for $R_2(n)$, all we have to do is to prove an upper bound for $R_1(n)$. In order to do so, we examine $R_1^{(1)}(n)$, $R_1^{(2)}(n)$, and $R_1^{(3)}(n)$, individually. \newline
\indent In \eqref{18} of Lemma~\ref{lemma3}, a simplified expression is given for $R_1^{(1)}(n)$. $\Sigma(n,Q)$ is evaluated in Lemma~\ref{lemma6}.
Similarly, $\epsilon(\beta)$ is defined to be either $1$ or $0$ depending on the existence of $\beta$. Hence, in order to prove an upper bound for $R_1^{(1)}(n)$, we must prove upper bounds for $\sum_{n=u+m}\frac{1}{2\sqrt{m}}$, $\Sigma(n,Q,\beta)$, and $u^{\beta-1}$. \newline
\indent It follows from Lemma~\ref{lemma7} that for all $N/2 < n \le N$, with $t = \tilde{r}$, except for $\ll N^{1-\delta_1}$ of these integers, when $\tilde{r} > N^{3.6\delta_1}$ \cite{5},
\begin{equation} \label{26}
\Sigma(n,Q,\beta) \ll n^{-2\delta_1}.
\end{equation}
From Lemma~\ref{lemma2}, it follows that for $Q < u \le N$,
\begin{equation} \label{27}
1-u^{\beta-1} = \int_{\beta}^{1}u^{s-1}\log u ds \ge c_3(1-\beta)\log n,
\end{equation}
which implies
\begin{equation} \label{28}
u^{\beta-1} \le 1-c_3(1-\beta)\log n.
\end{equation}
It is easy to see that
\begin{equation}
\sum_{n=u+m}\frac{1}{2\sqrt{m}} \le \frac{\zeta(2)}{2}.
\end{equation}
As a result, if the Siegel zero $\beta$ exists, then
\begin{equation} \label{29}
\begin{split}
R_1^{(1)}(n) &\ll \left(\sum_{n=u+m}\frac{1}{2\sqrt{m}}\right)\\
&\times \left(\mathscr{P}(n)+Bn^{-2\delta_1}-{n^{-2\delta_1}}(1-c_3(1-\beta)\log n)\right)+BN^{1/2-34\delta}.
\end{split}
\end{equation}
If the Siegel zero $\beta$ does not exist, then
\begin{equation} \label{30}
R_1^{(1)}(n) \ll \left(\sum_{n=u+m}\frac{1}{2\sqrt{m}}\right)\left(\mathscr{P}(n)+Bn^{-2\delta_1}\right)+ BN^{1/2-34\delta}.
\end{equation}
\indent We now move on to formulate an upper bound for $R_1^{(2)}(n)$. This is a rather easy task to complete since Polyakov \cite{5} already proves it. Using \eqref{14}, it follows from Lemma~\ref{lemma3},
\begin{equation} \label{31}
R_1^{(2)}(n) \ll N^{100\delta}.
\end{equation}
\indent Lastly, we prove an upper bound for $R_1^{(3)}(n)$. If we consider the sum in \eqref{22} of Lemma~\ref{lemma5} for $r \le N^{5\delta}$ then we can denote this first partial sum as $R_1^{(3.1)}(n)$. If we consider the same sum for $N^{5\delta} < r \le Q$, then we can denote this second partial sum as $R_1^{(3.2)}(n)$. Hence, $R_1^{(3)}(n) = R_1^{(3.1)}(n) + R_1^{(3.2)}(n)$. Polyakov \cite{5} proves that for all $n \le N$ except for $\ll N^{0.7}$ of these integers,
\begin{equation} \label{32}
R_1^{(3.1)}(n) \ll (\mathscr{P}(n)+Bn^{-2\delta_1})\sum_{r \le N^{5\delta}}\sum_{\chi^{*}\mod r}\left(\int_{-1/(r\tau)}^{1/(r\tau)}|W(\chi^{*},z)|^{2}dz\right)^{1/2}.
\end{equation}
The double sum in \eqref{32} was considered by Montgomery and Vaughan \cite{8} who showed that if the Siegel zero $\beta$ exists for $n \le N$,
\begin{eqnarray} \label{33}
\sum_{r \le N^{5\delta}}\sum_{\chi^{*}\mod r}\left(\int_{-1/(r\tau)}^{1/(r\tau)}|W(\chi^{*},z)|^{2}dz\right)^{1/2} \nonumber\\
\ll N^{1/2}\exp\Big\{\frac{-c}{\delta}\Bigr\}(1-\beta)\log N,
\end{eqnarray}
and in the absence of an exceptional zero $\beta$,
\begin{equation} \label{34}
\sum_{r \le N^{5\delta}}\sum_{\chi^{*}\mod r}\left(\int_{-1/(r\tau)}^{1/(r\tau)}|W(\chi^{*},z)|^{2}dz\right)^{1/2} \ll N^{1/2}\exp\Big\{\frac{-c}{\delta}\Bigr\}.
\end{equation}
With the assistance of Lemma~\ref{lemma7}, with $t=r$ and $\delta_1 = \delta$, Polyakov \cite{5} proves that independent of the existence of the Siegel zero $\beta$, for all $N/2 < n \le N$, except for $\ll N^{1-\delta}$ of these integers,
\begin{equation} \label{35}
R_1^{(3.2)}(n) \ll N^{1/2-2\delta}.
\end{equation}
Combining our results yields two upper bounds for $R(n)$ for $N/2 < n \le N$. The first bound assumes the existence of $\beta$ and holds for all but at most $\ll N^{1-\delta_1}$ of these integers,
\begin{equation} \label{36}
R(n) \ll_{N} \mathscr{P}(n)N^{1/2}\exp\Big\{\frac{-c}{\delta}\Bigr\}(1-\beta)\log N.
\end{equation}
Note that we only considered the case of $\tilde{r} > N^{3.6\delta_1}$ when obtaining the estimate in \eqref{36}. For $\tilde{r} \le N^{3.6\delta_1}$, we have then that for all $n \le N$ except for $\ll N^{0.7}$ of these integers, $|\Sigma(n,Q,\beta)| \le \mathscr{P}(n) + Bn^{-2\delta_1}$, from which one derives that the bound in \eqref{36} still holds. \newline
\indent The second bound assumes that $\beta$ does not exist, and holds for all $N/2 < n \le N$ except for at most $\ll N^{1-\delta}$ of these integers,
\begin{equation} \label{37}
R(n) \ll_{N} \mathscr{P}(n)N^{1/2}\exp\Big\{-\frac{c}{\delta}\Bigr\}.
\end{equation}
We can also obtain an unconditional upper bound on $\mathcal{R}(n)$ of the correct order of magnitude with no exceptions by way of sieve methods. We will derive an upper bound on $\mathcal{R}_k(n)$, the number of representations for the equation $n = p + m^k$ where $k \ge 2$, and then take the case for $k = 2$.\newline
\indent Let $\mathscr{A}$ stand for a general integer sequence to be ``sifted'' and let $\mathfrak{P}$ stand for a ``sifting'' set of primes. Moreover, $S(\mathscr{A}; \mathfrak{P}, z)$ is a sifting function where $z \ge 2$ is a real number. In the case of the present problem, we are sifting the set of numbers $n-m^2$ in order to estimate how often it is prime. The appropriate method we will utilize is to obtain a Selberg upper bound for $S(\mathscr{A}; \mathfrak{P}, z)$. Typically, an upper bound produced by Selberg's method is of $\ll$-type; however, by incorporating several important theorems of Halberstam and Richert \cite{12}, more explicit estimates can be yielded. We should note that neither the problem of the sum of a prime and a square nor the problem of the sum of a prime and a $k$-th power is dealt with in \cite{12}. \newline
\indent The sequence that is to be sifted for $1 < Y \le N$ is
\begin{equation*}
\mathscr{A} = \{n-m^k: N-Y<n \le N\}
\end{equation*}
From Theorem 5.3 of Halberstam and Richert \cite{12}, we let $F(n)$ to be a distinct irreducible polynomial with integral and positive leading coefficients, and let $\rho_k(p,n)$ denote the number of solutions to the congruence \newline
\begin{equation*}
m^k - n \equiv 0 \pmod p,
\end{equation*}
with $n$ constant for the purposes of the congruence.
Also, $N$ and $Y$ are real numbers satisfying \newline
\begin{equation*}
1 < Y \le N.
\end{equation*}
Hence from \cite{12} and taking $g = 1$,
\begin{equation*}
\begin{split}
|\{n: N-Y < n \le N,n-m^k = p\}|
&\le 2\prod_{p}\left(1-\frac{\rho_{k}(p,n)-1}{p-1}\right) \\
&\cdot \frac{Y}{\log Y}\left(1+O\left(\frac{\log\log 3Y}{\log Y}\right)\right).
\end{split}
\end{equation*}
\indent Now we take the case for $k=2$.
Thus for $Y = N$,
\begin{equation} \label{5}
\begin{split}
|\{n: 0 < n \le N,n-m^2 = p\}| &\le 2\prod_{p}\left(1-\frac{\left(\frac{n}{p}\right)}{p-1}\right)\\
&\cdot\frac{N}{\log N}\left(1+O\left(\frac{\log\log 3N}{\log N}\right)\right).
\end{split}
\end{equation}
\indent Theorem~\ref{result1} follows. Note that the $O$ constants may be subject to the dependency mentioned in Remark 1 of Halberstam and Richert \cite[Ch. 5, Sec. 6]{12}.

\section{A note on the combinatorial sieve}
\indent Of course, Selberg's method is not the only sieve method that can be used to obtain an upper bound on $\mathcal{R}(n)$ for $n \le N$ with no exceptions. For instance, we can obtain an upper bound for $S(\mathscr{A}; \mathfrak{P}, z)$ via a combinatorial sieve, as Zaccagnini \cite{15} mentions. Hence, as in the derivation of \eqref{5}, our upper bound will be applicable to $\mathcal{R}_k(n)$ and then we will take the case for $k = 2$. We proceed as follows. \newline
\indent The sequence that is to be sifted for $u \ge 1$ and $N^{1/u} \ge 2$ is
\begin{equation*}
\mathscr{A} = \{n-m^k: 0 \le n \le N\}
\end{equation*}
\begin{lem} \label{HR1}
Let $F_1(n), \cdots, F_g(n)$ be distinct irreducible polynomials with integral and positive leading coefficients where $F(n) = F_1(n) \cdots F_g(n)$. Let $\rho(p,n)$ denote the number of solutions to the congruence \newline
\begin{equation*}
F(n) \equiv 0 \pmod p.
\end{equation*}
Assume that for all primes $p$
\begin{equation*}
\rho(p,n) < p.
\end{equation*}
Let $q = q(N,u)$ denote a number having no prime divisors less than $N^{1/u}$ and satisfying $\frac{\log q}{\log N} \ll 1$. \newline
Then
\begin{equation} \label{4}
\begin{split}
|\{n: 0 \le n &\le N,\text{ $F_i(n) = q_i$ for $i = 1,\cdots,g$}\}| \\
&= N \prod_{p < N^{1/u}}\left(1-\frac{\rho(p,n)}{p}\right)\Big\{1+O_F\left(\exp\{-u(\log u \right. \\
&\quad \left. - \log\log 3u - \log g - 2)\}\right)+O_F\left(\exp\{-\sqrt{\log N}\}\right)\Bigr\}.
\end{split}
\end{equation}
\end{lem}
\begin{proof}
Lemma ~\ref{HR1} is Theorem 2.6 of Halberstam and Richert \cite{12}.
\end{proof}
\begin{rem} \label{r1}
The $O$ constant in \eqref{4} at most depends upon the degrees and coefficients of $F$.
\end{rem}
Note that as stated by Halberstam and Richert \cite{12}, the expression to the right side of the equality in \eqref{4} is equivalent to
\begin{equation} \label{rightside}
\begin{split}
& \left(u\cdot\exp\{-\gamma\}\right)^{g} \prod_{p}\left(1-\frac{\rho_{k}(p,n)-1}{p-1}\right)\left(1-\frac{1}{p}\right)^{-g+1}\frac{N}{\log N}\\
&\cdot \left(1+O\left(\exp\{-u(\log u - \log\log 3u - \log g - 2)\}\right)+O\left(\frac{u}{\log N}\right)\right).
\end{split}
\end{equation}
We take $g = 1$. Hence, we obtain from Lemma ~\ref{HR1} and \eqref{rightside} an upper bound on $\mathcal{R}_{k}(n)$ of the correct order of magnitude
\begin{equation*}
\begin{split}
|\{n: 0 < n &\le N,n-m^k = p\}|\\
&\le \left(u\cdot\exp\{-\gamma\}\right) \prod_{p}\left(1-\frac{\rho_{k}(p,n)-1}{p-1}\right)\frac{N}{\log N} \\
&\quad \times \left(1+O\left(\exp\{-u(\log u - \log\log 3u - 2)\}\right) \right. \left. +O\left(\frac{u}{\log N}\right)\right).
\end{split}
\end{equation*}
For $k = 2$ and $n \le N$ sufficiently large ($N^{1/u} \ge 2$ and $u \ge 1$), we have
\begin{equation}\label{combinatorial}
\begin{split}
\mathcal{R}(n) &\le \left(u\cdot\exp\{-\gamma\}\right) \mathscr{P}(n)\frac{N}{\log N}\\
&\cdot \left(1+O\left(\exp\{-u(\log u - \log\log 3u - 2)\}\right)+O\left(\frac{u}{\log N}\right)\right).
\end{split}
\end{equation}
As in \eqref{5}, the upper bound achieved in \eqref{combinatorial} has no exceptions for $n \le N$. Note that the $O$ constants may be subject to the dependency mentioned in Remark ~\ref{r1}.

\section*{Acknowledgements}
The author wishes to thank Alberto Perelli, Alessandro Languasco, Charles R. Greathouse IV, and the anonymous referee(s) for their assistance and helpful comments. The author is also grateful to David H. Low and Mohammad S. Moslehian for assisting with any typesetting issues.

\end{document}